\documentclass{amsart}
\usepackage{amssymb}

\theoremstyle{plain}
\newtheorem{theorem}{Theorem}[section]
\newtheorem{corollary}[theorem]{Corollary}
\newtheorem{definition}[theorem]{Definition}
\newtheorem{example}[theorem]{Example}

\newtheorem{lemma}[theorem]{Lemma}
\numberwithin{equation}{section}

\begin{document}

\title[Polynomial analogues]{Polynomial analogues of restricted multicolor 
$b$-ary partition functions}

\author{Karl Dilcher}
\address{Department of Mathematics and Statistics\\
         Dalhousie University\\
         Halifax, Nova Scotia, B3H 4R2, Canada}
\email{dilcher@mathstat.dal.ca}
\author{Larry Ericksen}
\address{P.O. Box 172, Millville, NJ 08332-0172, USA}
\email{LE22@cornell.edu}
\thanks{Research supported in part by the Natural Sciences and Engineering
        Research Council of Canada, Grant \# 145628481}

\date{}

\setcounter{equation}{0}

\begin{abstract}
Given an integer base $b\geq 2$, a number $\rho\geq 1$ of colors, and a finite
sequence $\Lambda=(\lambda_1,\ldots,\lambda_\rho)$ of positive integers, we
introduce the concept of a $\Lambda$-restricted $\rho$-colored $b$-ary 
partition of an integer $n\geq 1$. We also define a sequence of polynomials in
$\lambda_1+\cdots+\lambda_\rho$ variables, and prove that the $n$th polynomial
characterizes all $\Lambda$-restricted $\rho$-colored $b$-ary partitions of 
$n$. In the process we define a recurrence relation for the polynomials in 
question, obtain explicit formulas and identify a factorization theorem.
\end{abstract}

\maketitle

\section{Introduction}

The study of binary partitions and partition functions originated with Euler 
\cite[p.~162ff.]{Eu}, but is still an active area of research. These concepts
were later extended to $b$-ary and restricted $b$-ary partitions for arbitrary
integer bases $b\geq 2$; some of the most important contributions were made,
in chronological order, by Mahler \cite{Ma}, Churchhouse \cite{Ch}, 
R{\o}dseth \cite{Ro}, Andrews \cite{An}, Reznick \cite{Re}, and Dumont et al.
\cite{DST}. Further details, putting these publication in perspective, can be 
found in \cite{DE11}.

More recently, R{\o}dseth and Sellers \cite{RS} introduced and studied $b$-ary
overpartitions, in analogy to ordinary overpartitions that had been introduced
a little earlier by Corteel and Lovejoy \cite{CL}. 
Ordinary overpartitions were extended to restricted multicolor partitions by 
Keith \cite{Ke}, and colored $b$-ary partitions have recently been studied by 
Ulas and {\.Z}mija \cite{UZ}.

In almost all the papers mentioned above, the authors obtained interesting and
important arithmetic properties of the partition functions in question,
including Ramanujan-type congruences. The purpose of this paper is quite 
different. Building on earlier work, especially on our recent paper \cite{DE11},
we define a sequence of multivariable polynomials that characterize all
individual partitions of an integer $n\geq 0$ (as opposed to just their numbers)
for a large class of restricted colored $b$-ary partitions.

To make the results of the current paper more transparent and easier to 
understand, we recall some earlier work, which may serve as an example.
Given integers $b\geq 2$ and $\lambda\geq 1$, the representations
\begin{equation}\label{1.1}
n = \sum_{j\geq 0}c_jb^j,\qquad c_j\in\{0,1,\ldots,\lambda\},
\end{equation}
of an integer $n\geq 0$ are restricted $b$-ary partitions, and their
numbers are denoted by $S_b^\lambda(n)$ in \cite{DE11}. In particular, 
$S_b^{b-1}(n)=1$ since it corresponds to the unique $b$-ary representation of
$n$. In the base 2 case, $S_2^2(n)$ counts the number of hyperbinary 
representations of $n$, and $S_2^2(n)=s(n+1)$, where $\{s(n)\}$ is the 
well-known Stern sequence; see \cite[p.~470]{Re}. Related to this, Bates and
Mansour \cite{BM} and independently Stanley and Wilf \cite{SW} introduced
polynomial analogues as refinements of $s(n+1)$; their coefficients count 
hyperbinary representations with certain properties. 

Recently the current authors \cite{DE7} refined this further by introducing 
the following sequence of bivariate polynomials.

\begin{definition}[\cite{DE7}, Definition~4.1]\label{def:1.1}
Let $s$ and $t$ be fixed positive integer parameters. We define the 
$2$-parameter generalized Stern polynomials in the variables $y$ and $z$ by 
$\omega_{s,t}(0;y,z)=0$, $\omega_{s,t}(1;y,z)=1$, and for $n\geq 1$ by
\begin{align}
\omega_{s,t}(2n;y,z) &=  y\,\omega_{s,t}(n;y^s,z^t),\label{1.2} \\
\omega_{s,t}(2n+1;y,z) &= z\,\omega_{s,t}(n;y^s,z^t)+\omega_{s,t}(n+1;y^s,z^t).
\label{1.3}
\end{align} 
\end{definition}

By comparing the recurrence relations \eqref{1.2}, \eqref{1.3} with those of
Stern's sequence, it is clear that $\omega_{s,t}(n;1,1)=s(n)$, independent of
the parameters $s,t$. The usefulness of the polynomials $\omega_{s,t}(n;y,z)$,
also relevant for the current paper, lies in the following characterization of
all hyperbinary expansions of a positive integer $n$, which we denote by
${\mathbb H}_n$.

\begin{theorem}[\cite{DE7}, Theorem~4.2]\label{thm:1.2}
For any integer $n\geq 1$ we have
\begin{equation}\label{1.4}
\omega_{s,t}(n+1;y,z)=\sum_{h\in{\mathbb H}_n}y^{p_h(s)}z^{q_h(t)},
\end{equation}
where for each hyperbinary expansion $h$ of $n$, the exponents $p_h(s)$,
$q_h(t)$ are polynomials in $s$ and $t$ respectively, with only $0$ and $1$ 
as coefficients. Furthermore,~if 
\begin{align}
p_h(s) &= s^{\alpha_1}+\dots+s^{\alpha_\mu},\quad
0\leq\alpha_1<\dots<\alpha_\mu,\quad \mu\geq 0, \label{1.5} \\
q_h(t) &= t^{\beta_1}+\dots+t^{\beta_\nu},\quad
0\leq\beta_1<\dots<\beta_\nu,\quad \nu\geq 0, \label{1.6}
\end{align}
then the hyperbinary expansion $h\in{\mathbb H}_n$ is given by
\begin{equation}\label{1.7}
n = 2^{\alpha_1}+\dots+2^{\alpha_\mu}+(2^{\beta_1}+2^{\beta_1})+\dots
+(2^{\beta_\nu}+2^{\beta_\nu}).
\end{equation}
\end{theorem}

\begin{example}[\cite{DE7}, Example~8]\label{ex:1.3}
{\it With Definition~\ref{def:1.1} we compute}
\[
\omega_{s,t}(11;y,z) = z^{1+t^2}+y^sz^{t^2}+y^{s^2}z^{1+t}+y^{s^3}z+y^{s+s^3}.
\]
{\it If we choose the third term on the right, then according to 
Theorem~\ref{thm:1.2} the corresponding hyperbinary expansion will be 
$n=4+1+1+2+2$. The other four expansions can be ``read off" analogously.}
\end{example}

The sequence of polynomials $\omega_{s,t}(n;y,z)$ satisfies the following
generating function.

\begin{theorem}[\cite{DE7}, Theorem~4.4]\label{thm:1.4}
For integers $s,t\geq 1$ we have
\begin{equation}\label{1.8}
\sum_{n=1}^{\infty}\omega_{s,t}(n;y,z)q^n
= q\prod_{j=0}^{\infty}\left(1+y^{s^j}q^{2^j}+z^{t^j}q^{2^{j+1}}\right).
\end{equation}
\end{theorem}

Finally, there is an explicit formula, for which
we first need to recall some notation. We consider binomial coefficients 
modulo~2 as follows:
\[
\binom{n}{k}^*\equiv \binom{n}{k}\pmod{2},\qquad \binom{n}{k}^*\in\{0,1\}.
\]
Furthermore, let $n=\sum_{j\geq 0}c_j2^j$, $c_j\in\{0,1\}$, be the binary 
expansion of an integer $n\geq 0$. Then for integers $t\geq 1$ we define
\begin{equation}\label{1.9}
d_t(n):=\sum_{j\geq 0}c_jt^j.
\end{equation}
We can now state a first version of the desired explicit formula.

\begin{lemma}[\cite{DE7}, Theorem~4.3]\label{lem:1.5}
For integers $s,t\geq 1$ and $n\geq 0$ we have
\begin{equation}\label{1.10}
\omega_{s,t}(n+1;y,z)
=\sum_{k=0}^{\lfloor\frac{n}{2}\rfloor}\binom{n-k}{k}^*y^{d_s(n-2k)}z^{d_t(k)}.
\end{equation}
\end{lemma}

As shown in \cite{DE11}, the identity \eqref{1.10} can be rewritten in a form 
that will extend more easily to a more general situation.

\begin{theorem}[\cite{DE11}, Equation (55)]\label{thm:1.6}
For integers $s,t\geq 1$ and $n\geq 0$ we have
\begin{equation}\label{1.11}
\omega_{s,t}(n+1;y,z)
=\sum_{k_1+2k_2=n}\binom{k_1+k_2}{k_2}^*y^{d_s(k_1)}z^{d_t(k_2)}.
\end{equation}
\end{theorem}

Finally, the polynomials $\omega_{s,t}(n;y,z)$ also satisfy various identities,
an example of which will follow. For greater ease of notation we drop the 
subscripts $s, t$ and just write $\omega(n;y,z)$. We present a special case of 
Theorem~4.2 in \cite{DE10}.

\begin{theorem}\label{thm:1.7}
Let $\ell\geq 1$ and $k\geq 1$ be integers. Then for all integers $j$ with 
$0\leq j\leq 2^\ell$ we have 
\begin{align}
\omega(k\cdot 2^\ell+j;y,z)&=\omega(k+1;y^{s^\ell},z^{t^\ell})\cdot\omega(j;y,z)\label{1.12} \\
&\quad+\omega(k;y^{s^\ell},z^{t^\ell})\cdot\big(\omega(2^\ell+j;y,z)
-y^{s^\ell}\omega(j;y,z)\big).\nonumber
\end{align}
\end{theorem}

All this, from Definition~\ref{def:1.1} to Theorem~\ref{thm:1.7}, has been 
generalized, first to $b$-ary multivariable Stern polynomials in \cite{DE10}, 
and then further to polynomial analogues of restricted $b$-ary partition 
functions; see \cite{DE11}. For general $b\geq 2$, some phenomena become 
visible that are trivial in the case $b=2$. For instance, the general case of 
\eqref{1.12} leads to a factorization result that is meaningful only for 
$b\geq 3$; see Corollary~4.3 in \cite{DE10}.

It is the purpose of this paper to obtain results for polynomial analogues of
restricted colored $b$-ary partition functions, extending 
Theorems~\ref{thm:1.2}, \ref{thm:1.4}, \ref{thm:1.6}, and~\ref{thm:1.7}. We
begin by defining, in Section~2, the concepts of restricted colored $b$-ary
partitions, the corresponding partition functions, and their polynomial 
analogues. Section~3 contains the main result of this paper, namely a 
representation theorem that connects the restricted colored $b$-ary partitions
of Definition~\ref{def:2.1} with the multivariate polynomials of 
Definition~\ref{def:2.4}. In Section~4 we derive a recurrence relation for 
these polynomials, which we use in Section~5 to prove the representation 
theorem. In Sections~6 and~7 we obtain an explicit formula and a factorization
result, respectively, for the multivariate polynomials; we then end with some
further remarks and examples in Section~8.

\section{Colored $b$-ary partition functions}

The concepts of a restricted or unrestricted colored $b$-ary partition can be
seen as an extension of $b$-ary overpartitions that were introduced by 
R{\o}dseth and Sellers \cite{RS}. According to their definition, a $b$-ary
{\it overpartition} of an integer $n\geq 1$ is a sequence of nonnegative integer
powers whose sum is $n$, and where the first occurrence of a power $b$ may be
overlined.

This means that each $b$-ary partition can have two types (or ``colors") of
parts, namely the ``basic" type which may occur an arbitrary number of times, 
and an ``overlined" type, which may occur at most once. This leads to the
idea of assigning a fixed number $\rho\geq 1$ of colors in a more symmetric
fashion, that is, such that all colors are restricted.

\begin{definition}\label{def:2.1}
Let $b\geq 2$ and $\rho\geq 1$ be integers, and let
$\Lambda=(\lambda_1,\ldots,\lambda_\rho)$ be a finite integer sequence with
$1\leq\lambda_1\leq\cdots\leq\lambda_\rho$. Then a $\Lambda$-restricted 
$\rho$-colored $b$-ary partition of an integer $n\geq 1$ is a set
of powers $b^j$ whose sum is $n$, and where for a given integer 
$j\geq 0$, up to $\lambda_r$ parts $b^j$ can be assigned the $r$th color,
for $1\leq r\leq\rho$.
\end{definition}

In other words, each part $b^j$ can be assigned the $\ell$th color at most
$\lambda_\ell$ times, for $1\leq\ell\leq\rho$. 
If $C_b^\Lambda(n)$ denotes the number of partitions of $n$ as defined in 
Definition~\ref{def:2.1}, then the generating function is
\begin{equation}\label{2.1}
\sum_{n=0}^\infty C_b^\Lambda(n)q^n
= \prod_{j=0}^\infty\bigg(1+q^{b^j}+\cdots+q^{\lambda_1\cdot b^j}\bigg)\cdots
\bigg(1+q^{b^j}+\cdots+q^{\lambda_\rho\cdot b^j}\bigg). 
\end{equation}

We now illustrate Definition~\ref{def:2.1} with a concrete example.

\begin{example}\label{ex:2.2}
{\rm Let $b=2$ and $\Lambda=(2,3)$, and thus $\rho=2$. Then the 
$\Lambda$-restricted 2-colored binary partitions of $n=4$ are}

$4_2,\;4_1,\;2_2+2_2,\;2_2+2_1,\;2_1+2_1,\;2_2+1_2+1_2,\;2_2+1_2+1_1,\;
2_2+1_1+1_1,$

$2_1+1_2+1_2,\;2_1+1_2+1_1,\; 2_1+1_1+1_1,\;1_2+1_2+1_2+1_1,\;
1_2+1_2+1_1+1_1,$

\noindent
{\rm where the subscripts denote the colors of the part, and by convention we 
write each partition in non-increasing order of its parts, and assign the colors
also in non-increasing order. This is consistent with the notation of colored 
ordinary partitions used by Keith \cite{Ke}. Counting the partitions in this
example, we see that $C_2^\Lambda(4)=13$. On the other hand, by \eqref{2.1} we 
have}
\begin{align*}
\sum_{n=0}^\infty C_2^\Lambda(n)q^n
&= \prod_{j=0}^\infty\bigg(1+q^{2^j}+q^{2\cdot 2^j}\bigg)
\bigg(1+q^{2^j}+q^{2\cdot 2^j}+q^{3\cdot 2^j}\bigg) \\
&= 1+2q+5q^2+7q^3+13q^4+17q^5+26q^6+\ldots,
\end{align*}
{\rm which confirms that $C_2^\Lambda(4)=13$.}
\end{example}

When $\rho=1$, Definition~\ref{def:2.1} reduces to that of usual restricted 
$b$-ary partitions as in \eqref{1.1}. Their numbers were denoted by
$S_b^\lambda(n)$ in \cite{DE11}, with generating function
\begin{equation}\label{2.2}
\sum_{n=0}^\infty S_b^\lambda(n)q^n
= \prod_{j=0}^\infty\bigg(1+q^{b^j}+q^{2\cdot b^j}+\cdots+q^{\lambda\cdot b^j}\bigg).
\end{equation}
When $\Lambda=(\lambda)$, then for all $n\geq 0$ we obviously have 
$C_b^\Lambda(n)=S_b^\lambda(n)$. The generating function \eqref{2.2} was then
used in \cite{DE11} as the basis for the following definition of a polynomial
analogue.

\begin{definition}[\cite{DE11}, Definition~1]\label{def:2.3}
Let $b\geq 2$ and $\lambda\geq 1$ be integers, $T=(t_1,\ldots,t_\lambda)$ be
a $\lambda$-tuple of positive integer parameters, and $Z=(z_1,\ldots,z_\lambda)$
a $\lambda$-tuple of variables. Then we define the sequence of 
$\lambda$-variable polynomials $\omega_{b,T}^\lambda(n;Z)$ by the generating
function
\begin{equation}\label{2.3}
\sum_{n=0}^\infty \omega_{b,T}^\lambda(n+1;Z)q^n
=\prod_{j=0}^\infty\left(1+z_1^{t_1^j}q^{b^j}+z_2^{t_2^j}q^{2\cdot b^j}+\dots
+z_\lambda^{t_\lambda^j}q^{\lambda\cdot b^j}\right).
\end{equation}
\end{definition}

We immediately see that this extends both \eqref{1.8} and 
\eqref{2.2}. Indeed, if $T=(s,t)$, $Z=(y,z)$, and $b=\lambda=2$, then 
$\omega_{2,T}^2(n;Z)=\omega_{s,t}(n;y,z)$. Also, when 
$Z=(1,\ldots,1)$ and
$T$ is arbitrary, then for any $b\geq 2$ and $\lambda\geq 1$ we have
\[
\omega_{b,T}^\lambda(n+1;(1,\ldots,1)) = S_b^\lambda(n).
\]
With \eqref{2.1} and \eqref{2.3} in mind, we are now ready to define the main
object of study in this paper.

\begin{definition}\label{def:2.4}
Let $b\geq 2$ and $\rho\geq 1$ be integers, and let
$\Lambda=(\lambda_1,\ldots,\lambda_\rho)$ be a finite integer sequence with
$1\leq\lambda_1\leq\cdots\leq\lambda_\rho$. Furthermore, let 
$\lambda:=\lambda_1+\cdots+\lambda_\rho$ and let 
$T=(t_{1,1},\ldots,t_{1,\lambda_1};\ldots;t_{\rho,1},\ldots,t_{\rho,\lambda_\rho})$ 
be a $\lambda$-tuple of positive integer parameters, and 
$Z=(z_{1,1},\ldots,z_{1,\lambda_1};\ldots;z_{\rho,1},\ldots,z_{\rho,\lambda_\rho})$
a $\lambda$-tuple of variables. Then we define the sequence of
$\lambda$-variable polynomials $\Omega_{b,T}^\Lambda(n;Z)$ by the generating
function
\begin{align}
\sum_{n=0}^\infty \Omega_{b,T}^\Lambda(n;Z)q^n
=\prod_{j=0}^\infty&\left(1+z_{1,1}^{t_{1,1}^j}q^{b^j}+\dots
+z_{1,\lambda_1}^{t_{1,\lambda_1}^j}q^{\lambda_1\cdot b^j}\right)\cdots\label{2.4}\\
&\qquad\left(1+z_{\rho,1}^{t_{\rho,1}^j}q^{b^j}+\dots
+z_{\rho,\lambda_\rho}^{t_{\rho,\lambda_\rho}^j}q^{\lambda_\rho\cdot b^j}\right).\nonumber
\end{align}
\end{definition}

Suppose that in $Z$ we have $z_{i,j}=1$ for all relevant indices $i,j$. Then for
arbitrary $T$ we have
\begin{equation}\label{2.5}
\Omega_{b,T}^\Lambda(n;(1,\ldots,1)) = C_b^\Lambda(n),\qquad n=0,1,\ldots.
\end{equation}
When $\rho=1$, we set $\Lambda=(\lambda)$ and $t_{1,i}=t_i$, $z_{1,i}=z_i$ for
$1\leq i\leq\lambda$; then comparing \eqref{2.4} with \eqref{2.3} gives
\begin{equation}\label{2.6}
\Omega_{b,T}^\Lambda(n;Z) = \omega_{b,T}^\lambda(n+1;Z),\qquad \rho=1,\quad
n=0,1,\ldots.
\end{equation}
The argument $n+1$ on the right of this last identity is due to the fact that
the polynomials $\omega_{b,T}^\lambda(n+1;Z)$ came from the study of
generalized Stern polynomials; see also Definition~\ref{def:1.1}.

We conclude this section by illustrating Definition~\ref{def:2.4} with an
example that can be seen as a continuation of Example~\ref{ex:2.2}.

\begin{center}
\begin{table}[h]
{\renewcommand{\arraystretch}{1.5}
\begin{tabular}{|l||l|}
\hline
$n$  & $\Omega_{b,T}^{\Lambda}(n;Z)$ \\
\hline\hline
$0$    &  $1$ \\
\hline
$1$    &  $z_{1}+y_{1}$ \\
\hline
$2$    &  $z_{1}^{t_1}+z_{2}+z_{1} y_{1}+y_{1}^{s_1}+y_{2}$ \\
\hline
$3$    &  $z_{1}^{1+t_1}+z_{3}+z_{1}^{t_1} y_{1}+z_{2} y_{1}+z_{1} y_{1}^{s_1}+y_{1}^{1+s_1}+z_{1} y_{2}$ \\
\hline
$4$    &  $z_{1}^{t_1^2}+z_{1}^{t_1} z_{2}+z_{2}^{t_{2}}+z_{1}^{1+t_1} y_{1}+z_{3} y_{1}+z_{1}^{t_1} y_{1}^{s_1}$ \\
$$ & $+z_{2} y_{1}^{s_1}+y_{1}^{s_{1}^2}+z_{1} y_{1}^{1+s_1}+z_{1}^{t_1} y_{2}+z_{2} y_{2}+y_{1}^{s_1} y_{2}+y_{2}^{s_2}$ \\
\hline
$5$ & $y_1^{1+s_1^2}+y_1 y_2^{s_2}+y_1^{s_1^2} z_1+y_1^{s_1} y_2 z_1+y_2^{s_2} z_1+y_1^{1+{s_1}} z_1^{t_1}+y_1 z_1^{t_1^2}+y_1^{s_1} z_1^{1+{t_1}}$\\
$$ &  $+y_2 z_1^{1+t_1}+z_1^{1+t_1^2}+y_1^{1+{s_1}} z_2+y_1 z_1^{t_1} z_2+y_1 z_2^{t_2}+z_1 z_2^{t_2}+y_1^{s_1} z_3+y_2 z_3+z_1^{t_1} z_3$ \\
\hline
\end{tabular}}
\medskip
\caption{$\Omega_{b,T}^{\Lambda}(n;Z)$ as in Example~\ref{ex:2.5}, for 
$0\le n \le 5$.}
\label{tabpolJ01V12123P}
\end{table}
\end{center}
\vspace{-4ex}

\begin{example}\label{ex:2.5}
{\rm Let $b=2$ and $\Lambda=(2,3)$. To avoid double indices, we set
$T=(s_1,s_2;t_1,t_2,t_3)$ and $Z=(y_1,y_2;z_1,z_2,z_3)$. Then the generating
function \eqref{2.4} becomes}
\begin{equation}\label{2.7}
\sum_{n=0}^\infty\Omega_{b,T}^\Lambda(n;Z)q^n
=\prod_{j =0}^\infty\bigg(1+y_1^{s_1^j}q^{2^j}+y_2^{s_2^j}q^{2\cdot 2^j}\bigg)
\bigg(1+z_1^{t_1^j}q^{2^j} + z_2^{t_2^j}q^{2\cdot 2^j}
+ z_3^{t_3^j}q^{3 \cdot 2^j}\bigg).
\end{equation}

{\rm The first few polynomials, obtained by expanding the right-hand side of
\eqref{2.7}, are shown in Table~1.
Note that the number of terms of the polynomials in Table~1 are 1, 2, 5, 7, 13,
and 17, respectively. This is consistent with \eqref{2.5} and 
Example~\ref{ex:2.2}.}
\end{example}

We conclude this section with a result that is easily obtained by combining
Definitions~\ref{def:2.3} and~\ref{def:2.4}; it will be useful later, in 
Section~6.

\begin{theorem}\label{thm:2.6}
Let $b\geq 2$ and $\rho\geq 1$ be integers, and let $\Lambda$, $T$, and $Z$ be
as in Definition~\ref{def:2.4}. Furthermore, for each $\ell=1,\ldots,\rho$ let
\[
T_\ell=(t_{\ell,1},\ldots,t_{\ell,\lambda_\ell})\quad\hbox{and}\quad
Z_\ell=(z_{\ell,1},\ldots,z_{\ell,\lambda_\ell}).
\]
Then for each integer $n\geq 0$ we have
\begin{equation}\label{2.8}
\Omega_{b,T}^\Lambda(n;Z) = \sum_{n_1+\cdots+n_\rho=n}
\omega_{b,T_1}^{\lambda_1}(n_1+1;Z_1)\cdots
\omega_{b,T_\rho}^{\lambda_\rho}(n_\rho+1;Z_\rho).
\end{equation}
\end{theorem}

\begin{proof}
We use the fact that the right-hand side of \eqref{2.4} is the product of 
$\rho\geq 1$ generating functions of the type that occurs on the right-hand side
of the generating function \eqref{2.3}. Equating coefficients of $q^n$, we see
that the higher-order convolution on the right-hand side of \eqref{2.8} comes
from the product of $\rho$ appropriate power series from the left of 
\eqref{2.3}, while the left-hand side of \eqref{2.8} comes from the left of
the identity \eqref{2.4}.
\end{proof}

We note that in the special case $\rho=1$, the identity \eqref{2.8} reduces 
to \eqref{2.6}.

\section{A representation theorem}


The main reason for introducing polynomial analogues for the Stern and $b$-ary 
Stern sequences \cite{DE7, DE9} is the fact that they can be used to 
characterize all hyperbinary and hyper $b$-ary representations. 
Theorem~\ref{thm:1.2} and Example~\ref{ex:1.3} may serve as an illustration 
of this. More generally, in \cite{DE11} we used polynomial analogues for 
restricted $b$-ary partition functions to characterize all restricted $b$-ary 
partitions of a given integer.

In the present section we will extend these representation theorems to
restricted {\it colored} $b$-ary partitions. We begin with an example.

\begin{example}\label{ex:3.1}
{\rm As in Example~\ref{ex:2.2} we let $b=2$ and $\Lambda=(2,3)$. This time we
consider the $\Lambda$-restricted 2-colored binary partitions of $n=3$, namely}
\begin{equation}\label{3.1}
2_2+1_2, 2_2+1_1, 2_1+1_2, 2_1+1_1, 1_2+1_2+1_2, 1_2+1_2+1_1, 1_2+1_1+1_1.
\end{equation}
{\rm On the other hand, from Table~1 we have}
\begin{equation}\label{3.2}
\Omega_{2,T}^\Lambda(3;Z) = z_{1}^{1+t_1}+z_{3}+z_{1}^{t_1} y_{1}+z_{2} y_{1}
+z_{1} y_{1}^{s_1}+y_{1}^{1+s_1}+z_{1} y_{2}.
\end{equation}
{\rm We rewrite this polynomial in the notation of Definition~\ref{def:2.4},
namely with $T=(t_{1,1},t_{1,2};t_{2,1},t_{2,2},t_{2,3})$ and
$Z=(z_{1,1},z_{1,2};z_{2,1},z_{2,2},z_{2,3})$. Then \eqref{3.2} becomes}
\begin{equation}\label{3.3}
\Omega_{2,T}^\Lambda(3;Z) = z_{2,1}^{1+t_{2,1}}+z_{2,3}+z_{2,1}^{t_{2,1}}z_{1,1}
+z_{2,2}z_{1,1}+z_{2,1}z_{1,1}^{t_{1,1}}+z_{1,1}^{1+t_{1,1}}
+z_{2,1}z_{1,2}.
\end{equation}
{\rm We assign each of the 
$C_2^\Lambda(3)=7$ partitions in \eqref{3.1} to one of the seven monomials in
\eqref{3.3} as shown in Table~2.}
\begin{center}
\begin{table}[h]
{\renewcommand{\arraystretch}{1.3}
\begin{tabular}{|c|c|c|}
\hline
{\rm Monomial} & {\rm Partition} & {\rm Rewritten} \\
\hline
$z_{2,1}^{1+t_{2,1}}$ & $1\cdot 2_2^0+1\cdot 2_2^1$ & $2_2+1_2$ \\
$z_{2,3}$ & $3\cdot 2_2^0$ & $1_2+1_2+1_2$ \\
$z_{2,1}^{t_{2,1}}z_{1,1}$ & $1\cdot 2_2^1+1\cdot 2_1^0$ & $2_2+1_1$ \\
$z_{2,2}z_{1,1}$ & $2\cdot 2_2^0+1\cdot 2_1^0$ & $1_2+1_2+1_1$ \\
$z_{2,1}z_{1,1}^{t_{1,1}}$ & $1\cdot 2_2^0+1\cdot 2_1^1$ & $2_1+1_2$ \\
$z_{1,1}^{1+t_{1,1}}$ & $1\cdot 2_1^0+1\cdot 2_1^1$ & $2_1+1_1$ \\
$z_{2,1}z_{1,2}$ & $1\cdot 2_2^0+2\cdot 2_1^0$ & $1_2+1_1+1_1$ \\
\hline
\end{tabular}}
\medskip
\caption{The monomials in \eqref{3.3}.}
\end{table}
\end{center}
\vspace{-8ex}
\end{example}
The one-to-one correspondence shown in Table~2 is a special case of the 
following representation theorem. Recall that $C_b^\Lambda(n)$ is the number
of $\Lambda$-restricted colored $b$-ary partitions of $n$; see 
Definition~\ref{def:2.1}.


\begin{theorem}\label{thm:3.2}
Let $b\geq 2$ and $\rho\geq 1$ be integers, and let the finite sequences 
$\Lambda$, $T$, and $Z$ be as in Definition~\ref{def:2.4}. 

\smallskip
\noindent
$(a)$ For any integer $n\geq 1$ we have
\begin{equation}\label{3.4}
\Omega_{b,T}^\Lambda(n;Z) 
= \sum_{h=1}^{C_b^\Lambda(n)}\prod_{\ell=1}^\rho
z_{\ell,1}^{p_{h,\ell,1}(t_{\ell,1})}\cdots 
z_{\ell,\lambda_\ell}^{p_{h,\ell,\lambda_\ell}(t_{\ell,\lambda_\ell})},
\end{equation}
where for each $h=1,\ldots,C_b^\Lambda(n)$, the exponents
$p_{h,\ell,i}(t_{\ell,i})$, $1\leq\ell\leq\rho$ and $1\leq i\leq\lambda_\ell$, 
are polynomials with only $0$ and $1$ as coefficients.

\smallskip
\noindent
$(b)$ We fix an $h=1,\ldots, C_b^\Lambda(n)$, and for each pair $(\ell,i)$ with
$1\leq\ell\leq\rho$ and $1\leq i\leq\lambda_\ell$ we write
\begin{equation}\label{3.5}
p_{h,\ell,i}(t_{\ell,i}) = t_{\ell,i}^{\tau_{\ell,i}(1)}
+t_{\ell,i}^{\tau_{\ell,i}(2)}+\cdots+t_{\ell,i}^{\tau_{\ell,i}(\mu_{\ell,i})},
\quad \mu_{\ell,i}\geq 0.
\end{equation}
Let the exponents $\tau_{\ell,i}(j)$, $j=1,2,\ldots,$ form a strictly 
increasing finite sequence of nonnegative integers, and by convention we let
$\mu_{\ell,i}=0$ correspond to $p_{h,\ell,i}(t_{\ell,i})=0$. Then the 
$\Lambda$-restricted colored $b$-ary partition of $n$ corresponding to the 
index-$h$ monomial on the right-hand side of \eqref{3.4} is
\begin{equation}\label{3.6}
\sum_{\ell=1}^{\rho}\sum_{i=1}^{\lambda_\ell}
\bigg(i\cdot b_{\ell}^{\tau_{\ell,i}(1)}+\cdots
+i\cdot b_{\ell}^{\tau_{\ell,i}(\mu_{\ell,i})}\bigg),
\end{equation}
where the subscript $\ell$ in $b_\ell$ indicates the $\ell$th color of the
part in question.
\end{theorem}

When $\rho=1$, Theorem~\ref{thm:3.2} becomes a special case of Theorem~4 in
\cite{DE11}. By further restricting our attention to $\lambda_1=\lambda=b$, we
obtain Theorem~11 in \cite{DE9}; the case $b=2$ of this is presented as
Theorem~\ref{thm:1.2} above.

For the proof of Theorem~\ref{thm:3.2} we require recurrence relations, which
will be derived in the next section. We conclude the present section with an 
example.

\begin{example}\label{ex:3.3}
{\rm Once again we take $b=2$ and $\Lambda=(2,3)$. From among the 17 terms of
the entry for $\Omega_{2,T}^{\Lambda}(5;Z)$ in Table~1 we choose two monomials 
and rewrite them in terms of $T=(t_{1,1},t_{1,2};t_{2,1},t_{2,2},t_{2,3})$ and
$Z=(z_{1,1},z_{1,2};z_{2,1},z_{2,2},z_{2,3})$, as in Example~\ref{ex:3.1}.

First we consider
\[
y_1^{1+s_1^2} = z_{1,1}^{1+t_{1,1}^2}.
\]
Using \eqref{3.5} and \eqref{3.6}, we see that this monomial corresponds to
the colored binary partition $1\cdot 2_1^0+1\cdot 2_1^2$, which we rewrite as 
$4_1+1_1$.

Next we consider the somewhat more interesting monomial
\[
y_1z_1^{t_1}z_2 = z_{1,1}\cdot z_{2,1}^{t_{2,1}}\cdot z_{2,2}.
\]
Using \eqref{3.5} and \eqref{3.6} again, we get the partition
$1\cdot 2_1^0+1\cdot 2_2^1+2\cdot 2_2^0$, which we rewrite as $2_2+1_2+1_2+1_1$.

Finally, we note that the correspondence between \eqref{3.5} and \eqref{3.6} 
also leads to the entries of Table~2.}
\end{example}

\section{Recurrence relations}

In this section we generalize the recurrence relations given by \eqref{1.2} and
\eqref{1.3}, and show that the polynomials $\Omega_{b,T}^\Lambda(n;Z)$ satisfy
similar, though vastly extended recurrence relations. In the case of polynomial
analogues of ordinary restricted $b$-ary partition functions, i.e., when 
$\rho=1$, we obtained such recurrence relations in \cite{DE11}. Rewriting the
relevant theorem in terms of the polynomials $\Omega_{b,T}^\Lambda(n;Z)$ by 
using the identity \eqref{2.6}, we have the following result.

\begin{theorem}[\cite{DE11}, Theorem~5]\label{thm:4.1}
Let $b\geq 2$ and $\lambda\geq 1$ be integers, and set $\Lambda=(\lambda)$, so
that $\rho=1$. Furthermore, let
\[
Z=(z_1,\ldots,z_\lambda),\quad T=(t_1,\ldots,t_\lambda),\quad
Z^T=(z_1^{t_1},\ldots,z_\lambda^{t_\lambda}).
\]
Then $\Omega_{b,T}^\Lambda(0;Z)=1$, and for integers $n\geq 0$ we have
\begin{equation}\label{4.1}
\Omega_{b,T}^\Lambda(bn+j;Z) = \sum_{k=0}^{\lfloor\lambda/b\rfloor}
z_{kb+j}\Omega_{b,T}^\Lambda(n-k;Z^T)\quad (j=0,\ldots,b-1),
\end{equation}
with the conventions that $z_0=1$, $z_\mu=0$ for $\mu>\lambda$, and
$\Omega_{b,T}^\Lambda(m;Z)=0$ if $m<0$.
\end{theorem}

\begin{example}\label{ex:4.2}
{\rm Let $\lambda=b=2$. Then, writing out the cases $j=0$ and $j=1$ separately,
the identity \eqref{4.1} reduces to
\begin{align}
\Omega_{2,T}^\Lambda(2n;z_1,z_2)&=z_0\Omega_{2,T}^\Lambda(n;z_1^{t_1},z_2^{t_2})
+z_2\Omega_{2,T}^\Lambda(n-1;z_1^{t_1},z_2^{t_2}),\label{4.2}\\
\Omega_{2,T}^\Lambda(2n+1;z_1,z_2)&=z_1\Omega_{2,T}^\Lambda(n;z_1^{t_1},z_2^{t_2})
+z_3\Omega_{2,T}^\Lambda(n-1;z_1^{t_1},z_2^{t_2})\label{4.3}.
\end{align}
In this case, i.e., when $\Lambda=(2)$, we have 
$\Omega_{2,T}^\Lambda(n;z_1,z_2)=\omega_{t_1,t_2}(n+1;z_1,z_2)$, and also
$z_0=1$ and $z_3=0$. Therefore \eqref{4.2} and \eqref{4.3} are identical with
\eqref{1.3} and \eqref{1.2}, respectively.}
\end{example}

We now extend Theorem~\ref{thm:4.1} to polynomial analogues of 
$\Lambda$-restricted {\it colored} $b$-ary partition functions.

\begin{theorem}\label{thm:4.3}
Let $b\geq 2$ and $\rho\geq 1$ be integers, and let the finite sequences
$\Lambda$, $T$, and $Z$ be as in Definition~\ref{def:2.4}. 
Then $\Omega_{b,T}^\Lambda(0;Z)=1$, and for integers $n\geq 0$ we have
\begin{equation}\label{4.4}
\Omega_{b,T}^\Lambda(bn+j;Z) 
= \sum_{k=0}^{\lfloor\lambda/b\rfloor}Y_{bk+j}\Omega_{b,T}^\Lambda(n-k;Z^T)
\qquad (j=0,\ldots b-1),
\end{equation}
with the convention that $\Omega_{b,T}^\Lambda(m;Z)=0$ if $m<0$. The 
coefficients $Y_\nu$, $0\leq\nu\leq\lambda$, are given by 
\begin{equation}\label{4.5}
Y_\nu = \sum z_{1,i_1}z_{2,i_2}\cdots z_{\rho,i_\rho},\qquad
z_{1,0}=z_{2,0}=\cdots=z_{\rho,0}=1,
\end{equation}
with the sum taken over all $i_1,i_2,\ldots,i_\rho$ with $i_1+\cdots+i_\rho=\nu$
and $0\leq i_\ell\leq\lambda_\ell$ for $\ell=1,\ldots,\rho$. In particular,
$Y_0=1$, $Y_1=z_{1,1}+z_{2,1}+\cdots+z_{\rho,1}$, 
$Y_\lambda = z_{1,\lambda_1}z_{2,\lambda_2}\cdots z_{\rho,\lambda_\rho}$, and
$Y_\mu=0$ when $\mu>\lambda$.
\end{theorem}

When $\rho=1$, we see that $Y_\nu=z_{1,\nu}=z_\nu$ for $0\leq\nu\leq\lambda$,
and so Theorem~\ref{thm:4.3} reduces to Theorem~\ref{thm:4.1}. Before proving
Theorem~\ref{thm:4.3}, we present an example.

\begin{example}\label{ex:4.4}
{\rm Continuing with Examples~\ref{ex:2.5} and~\ref{ex:3.1}, we let $b=2$ and
$\Lambda=(2,3)$, with $Z=(y_1,y_2;z_1,z_2,z_3)$. For the sake of simplicity
we delete the subscripts and superscripts of $\Omega$; Theorem~\ref{thm:4.3}
then gives
\begin{align}
\Omega(2n;Z) 
&= Y_{0}\Omega(n;Z^T) + Y_{2}\Omega(n-1;Z^T) +Y_{4}\Omega(n-2;Z^T),\label{4.6}\\
\Omega(2n+1;Z) &=   
Y_{1}\Omega(n;Z^T) + Y_{3}\Omega(n-1;Z^T) + Y_{5}\Omega(n-2;Z^T),\label{4.7}
\end{align}
where $Y_0=1$, $Y_1=y_1+z_1$, $Y_2=z_2+y_1z_1+y_2$, $Y_3=z_3+y_1z_2+y_2z_1$,
$Y_4=y_1z_3+y_2z_2$, and $Y_5=y_2z_3$.
As a specific example we take \eqref{4.7} with $n=2$. Then
\begin{align*}
\Omega(5;Z) &= (z_1+y_1)\cdot\Omega(2;Z^T)+(z_3+y_1z_2+y_2z_1)\cdot\Omega(1;Z^T)
+y_2z_3\cdot\Omega(0;Z^T)\\
&= (z_1+y_1)\cdot(z_1^{t_1^2}+z_2^{t_2}+z_1^{t_1}y_1^{s_1}+y_1^{s_1^2}+y_2^{s_2})\\
&\qquad+(z_3+y_1z_2+y_2z_1)\cdot(z_1^{t_1}+y_1^{s_1})+y_2z_3,
\end{align*}
where we have used the first three entries in Table~1. Expanding the right-hand
side of this last identity, we obtain the entry for $n=5$ in Table~1.}
\end{example}

\begin{proof}[Proof of Theorem~\ref{thm:4.3}]
For ease of notation we delete subscripts and superscripts from $\Omega$.
By manipulating the infinite product in \eqref{2.4}, we get
\begin{align*}
\sum_{m=0}^\infty\Omega(m;Z)q^m
&=\prod_{j=0}^\infty\prod_{\ell=1}^\rho
\left(1+z_{\ell,1}^{t_{\ell,1}^j}q^{b^j}+\cdots
+z_{\ell,\lambda_\ell}^{t_{\ell,\lambda_\ell}^j}q^{\lambda_\ell\cdot b^j}\right)\\
&=\prod_{\ell=1}^\rho\left(
1+z_{\ell,1}q+\cdots+z_{\ell,\lambda_\ell}q^{\lambda_\ell}\right)\\
&\qquad\cdot\prod_{j=0}^\infty\prod_{\ell=1}^\rho\left(
1+z_{\ell,1}^{t_{\ell,1}^{j+1}}q^{b^{j+1}}+\cdots
+z_{\ell,\lambda_\ell}^{t_{\ell,\lambda_ell}^{j+1}}q^{\lambda_\ell\cdot b^{j+1}}\right)\\
&=\big(1+Y_1 q+\cdots+Y_\lambda q^\lambda\big) \\
&\qquad\cdot\prod_{j=0}^\infty\prod_{\ell=1}^\rho\left(
1+(z_{\ell,1}^{t_{\ell,1}})^{t_{\ell,1}^j}(q^b)^{b^j}+\cdots
+(z_{\ell,\lambda_\ell}^{t_{\ell,\lambda}})^{t_{\ell,\lambda_\ell}^j}
(q^b)^{\lambda_\ell\cdot b^j}\right).
\end{align*}
We then obtain, again with \eqref{2.4},
\begin{equation}\label{4.8}
\sum_{m=0}^\infty\Omega(m;Z)q^m
=\big(1+Y_1 q+\cdots+Y_\lambda q^\lambda\big)
\sum_{\nu=0}^\infty\Omega(\nu;Z^T)q^{\nu b}. 
\end{equation}
With the conventions $Y_0=1$ and $Y_\mu=0$ for $\mu>\lambda$, we get
\begin{equation}\label{4.9}
1+Y_1 q+\cdots+Y_\lambda q^\lambda
= \sum_{k=0}^{\lfloor\lambda/b\rfloor}\sum_{j=0}^{b-1}Y_{kb+j} q^{kb+j}
= \sum_{k=0}^{\lfloor\lambda/b\rfloor}q^{kb}\sum_{j=0}^{b-1}Y_{kb+j}q^j.
\end{equation}
Next we write $m=bn+j$, $0\leq j\leq b-1$, in \eqref{4.8} and substitute
\eqref{4.9} into \eqref{4.8}. Then \eqref{4.8} becomes
\begin{align}
\sum_{n=0}^\infty\sum_{j=0}^{b-1}&\Omega(bn+j;Z)q^{bn+j}\label{4.10}\\
&=\bigg(\sum_{k=0}^{\lfloor\lambda/b\rfloor}q^{kb}\sum_{j=0}^{b-1}Y_{bk+j}q^j\bigg)
\cdot\sum_{\nu=0}^\infty\Omega(\nu;Z^T)q^{\nu b}\nonumber \\
&=\sum_{n=0}^\infty\sum_{j=0}^{b-1}
\bigg(\sum_{k=0}^{\lfloor\lambda/b\rfloor}Y_{bk+j}\Omega(n-k;Z^T)\bigg)q^{nb+j},\nonumber
\end{align}
where we have used a Cauchy product, with $k+\nu=n$. Finally, by equating
coefficients of $q^{nb+j}$ in the first and the third line of \eqref{4.10},
we immediately get the desired identity \eqref{4.4}.
\end{proof}

We finish this section by explicitly stating a functional equation that was 
derived as part of the proof of Theorem~\ref{thm:4.3}. We set
\begin{equation}\label{4.11}
F(Z,q) := \sum_{m=0}^{\infty}\Omega_{b,T}^\Lambda(m;Z);
\end{equation}
then the identity \eqref{4.8} can be rewritten as follows.

\begin{corollary}\label{cor:4.5}
Let $b\geq 2$ and $\rho\geq 1$ be integers, let the finite sequences $\Lambda$,
$T$, and $Z$ be as in Definition~\ref{def:2.4}, and the polynomials $Y_\nu$,
$1\leq\nu\leq\lambda$, as in \eqref{4.5}. Then the generating function $F(Z,q)$,
defined by \eqref{4.11}, satisfies the functional equation
\[
F(Z,q)=\big(1+Y_1 q+\cdots+Y_\lambda q^\lambda\big)F(Z^T,q^b).
\]
\end{corollary}

\section{Proof of Theorem~3.2}

In this section we use the recurrence relation, Theorem~\ref{thm:4.3}, to prove
our representation result, namely Theorem~\ref{thm:3.2}. In doing so, we follow
the outline of the proof of Theorem~4 in \cite{DE11}, although the details will
be different. In order to simplify terminology, if there is no danger of 
ambiguity we will refer to a $\Lambda$-restricted $\rho$-colored $b$-ary 
partition, as defined in Definition~\ref{def:2.1}, simply as ``a partition 
of $n$".

We fix the integers $b\geq 2$ and $\rho\geq 1$, as well as the $\rho$-tuple
$\Lambda$, integer $\lambda$, and the $\lambda$-tuples $T$ and $Z$ as in
Definition~\ref{def:2.4}. As we did in the proof of Theorem~\ref{thm:4.3}, we
suppress the subscripts and superscripts of $\Omega$. We proceed by induction
on $n$.

\medskip
{\bf 1.} For the induction beginning we consider \eqref{4.4} with $n=0$. Then
\begin{equation}\label{5.1}
\Omega(j;Z) = Y_j,\qquad j=0,1,\ldots,b-1,
\end{equation}
where by \eqref{4.5} we have
\begin{equation}\label{5.2}
Y_j = \sum_{\substack{i_1+\cdots+i_\rho=j\\0\leq i_\ell\leq\lambda_\ell,\,\ell=1,\ldots,\rho}}
z_{1,i_1}z_{2,i_2}\cdots z_{\rho,i_\rho}.
\end{equation}
On the other hand, the only partitions of $j$, with $0\leq j\leq b-1$, are
\begin{equation}\label{5.3}
i_1b_1^0+i_2b_2^0+\cdots+i_\rho b_\rho^0,\qquad 
0\leq i_\ell\leq\lambda_\ell,\;\ell=1,\ldots,\rho.
\end{equation}
This is consistent with \eqref{3.4}--\eqref{3.6}. Indeed, in view of 
\eqref{5.2}, the identity \eqref{3.4} becomes
\[
\Omega(j;Z)
= \sum_{h=1}^{C_b^\Lambda(n)}\prod_{\ell=1}^\rho
z_{\ell,i_\ell}^{p_{h,\ell,i_\ell}(t_{\ell,i_\ell})}.
\]
This means that, in view of \eqref{3.5}, for each $\ell=1,\ldots,\rho$ we have
\[
p_{h,\ell,i_\ell}(t_{\ell,i_\ell})=1,\quad \mu_{\ell,i_\ell}=1,\quad
\tau_{\ell,i_\ell}=0,
\]
and so by \eqref{3.6} the corresponding partition of $j$ is
\[
\sum_{\ell=1}^\rho i_\ell b_\ell^0,
\]
which is consistent with \eqref{5.3}. This completes the induction beginning.

\medskip
{\bf 2.} Next we assume that Theorem~\ref{thm:3.2} is true for all 
$m\leq nb-1$, for some $n\geq 1$. Our goal is to show that it is also true for 
$m=nb+j$, for all $j=0,1,\ldots,b-1$. To do so, we first fix an arbitrary $j$
with $0\leq j\leq b-1$. We then consider all partitions of $nb+j$, which can be
obtained recursively as follows. We fix an integer $k\geq 0$ and
\begin{enumerate}
\item[(a)] Take all partitions of $n-k$ and multiply each power of $b_\ell$
(including the 0th power) by $b_\ell$, $\ell=1,2,\ldots,\rho$, thus raising all
powers by 1.
\item[(b)] To each of the``raised" partitions from (a), add all partitions of
$kb+j$ of the form
\begin{equation}\label{5.4}
kb+j = i_1b_1^0+i_2b_2^0+\cdots+i_\rho b_\rho^0,\qquad
0\leq i_\ell\leq\lambda_\ell\quad\hbox{for}\quad\ell=1,\ldots,\rho.
\end{equation}
\item[(c)] Do (a) and (b) for all $k\geq 0$ that satisfy $kb+j\leq\lambda$.
\end{enumerate}

This procedure gives all partitions of $nb+j$ since $(n-k)\cdot b+(kb+j)=nb+j$.
Also, the maximal $k$ given by (c) is $\lfloor\lambda/b\rfloor$.

\medskip
{\bf 3.} Using the induction hypothesis and \eqref{3.4}, we have
\begin{equation}\label{5.5}
\Omega(n-k;Z)
= \sum_{h=1}^{C_b^\Lambda(n-k)}\prod_{\ell=1}^\rho
z_{\ell,1}^{p_{h,\ell,1}(t_{\ell,1})}\cdots
z_{\ell,\lambda_\ell}^{p_{h,\ell,\lambda_\ell}(t_{\ell,\lambda_\ell})},
\end{equation}
with exponents 
$p_{h,\ell,1}(t_{\ell,1}),\ldots,p_{h,\ell,\lambda_\ell}(t_{\ell,\lambda_\ell})$
as in \eqref{3.5}. In order to lift the partitions of $n-k$ to those of 
$(n-k)b$, which corresponds to step (a), all powers of $t_{\ell,i}$,
$1\leq\ell\leq\rho$ and $1\leq i\leq\lambda_\ell$, are augmented by 1. This,
in turn, means that \eqref{5.5} changes to $\Omega(n-k;Z^T)$.

Next, for each $k$ with $0\leq k\leq\lfloor\lambda/b\rfloor$, we multiply
$\Omega(n-k;Z^T)$ by the monomial
\[
z_{1,i_1}z_{2,i_2}\cdots z_{\rho,i\rho},\qquad 
0\leq i_{\ell}\leq\lambda_{\ell}\quad\hbox{for}\quad\ell=1,\ldots,\rho,
\]
which corresponds to an individual partition of $kb+j$ of the form \eqref{5.4}
in step (b). This means that for each $\ell=1,\ldots,\rho$, the exponents of
$z_{\ell,i_\ell}$ are increased by 1, which is again consistent with the
duality between monomials and individual partitions. Now, taking into account
all partitions of $kb+j$ of the form \eqref{5.4} means multiplying 
$\Omega(n-k;Z^T)$ by
\[
\sum_{\substack{i_1+\cdots+i_\rho=kb+j\\0\leq i_\ell\leq\lambda_\ell,\,\ell=1,\ldots,\rho}}
z_{1,i_1}z_{2,i_2}\cdots z_{\rho,i_\rho}.
\]
But, by \eqref{4.5}, this sum is exactly $Y_{kb+j}$.

\medskip
{\bf 4.} Recall that steps 2 and 3 were done for an arbitrary fixed integer
$k$ with $0\leq k\leq \lfloor\lambda/b\rfloor$. The polynomial corresponding to
all partitions of $nb+j$ is therefore
\[
\sum_{k=0}^{\lfloor\lambda/b\rfloor}Y_{bk+j}\Omega(n-k;Z^T),
\]
which, by Theorem~\ref{thm:4.3}, is $\Omega(bn+j;Z)$. Since $j$ was an 
arbitrarily chosen but fixed integer with $0\leq j\leq b-1$, this concludes
the proof of Theorem~\ref{thm:3.2} by induction.

\section{Explicit formulas}

In this section we derive explicit formulas for our polynomials
$\Omega_{b,T}^\Lambda(n;Z)$, thus generalizing Theorem~\ref{thm:1.6}. For
greater clarity, we begin with an intermediate result for the binary case, i.e.,
the case $b=2$. The function $d_t(n)$ is defined as in \eqref{1.9}, and the
``starred" multinomial coefficient is defined to be the least nonnegative
residue modulo 2 of the original multinomial coefficient.

\begin{theorem}[\cite{DE11}, Corollary~17]\label{thm:6.1}
Let $\lambda\geq 1$ be an integer, and $Z$ and $T$ as in 
Definition~\ref{def:2.3}. Then for $n\geq 0$ we have
\begin{equation}\label{6.1}
\omega_{2,T}^\lambda(n+1;Z) = \sum_{k_1+2k_2+\cdots+\lambda k_{\lambda}=n}
\binom{k_1+\cdots+k_\lambda}{k_1,\ldots,k_\lambda}^*
z_1^{d_{t_1}(k_1)}\cdots z_\lambda^{d_{t_\lambda}(k_\lambda)}.
\end{equation}
\end{theorem}

We see that for $\lambda=2$ the identity \eqref{6.1} reduces to \eqref{1.11}.
On the other hand, Theorem~\ref{thm:6.1} was extended in \cite{DE11} to
arbitrary bases $b\geq 2$. Since this more general case will be needed to
obtain the main result of this section, we state it here, but without proof. 
We need some additional definitions with corresponding notations; see also
\cite{DE11}.

First we define the set
\begin{equation}\label{6.2}
M_b := \left\{\sum_{j\geq 0}c_jb^j\mid c_j\in\{0,1\}\right\},
\end{equation}
i.e., the set of all nonnegative integers whose $b$-ary digits are only 0 or 1.
Clearly we have $M_2={\mathbb N}\cup\{0\}$.

Next we extend \eqref{1.9} as follows. Let $k\in M_b$ with 
$k=\sum_{j\geq 0}c_jb^j$, $c_j\in\{0,1\}$. Then for an integer base $t\geq 1$ 
we define
\begin{equation}\label{6.3}
d_t^b(k):=\sum_{j\geq 0}c_jt^j.
\end{equation}
It is clear that $d_t^2(k)=d_t(k)$ and $d_b^b(k)=k$ for all integers $b\geq 2$
and $k\in M_b$.

Finally, we extend the multinomial coefficient modulo 2, as it is used in 
Theorem~\ref{thm:6.1}. If $k_1,\ldots,k_\lambda\in M_b$, we set
\begin{equation}\label{6.4}
\binom{k_1+\cdots+k_\lambda}{k_1,\ldots,k_\lambda}_b^*
:=\binom{d_2^b(k_1)+\cdots+d_2^b(k_\lambda)}{d_2^b(k_1),\ldots,d_2^b(k_\lambda)}^*,
\end{equation}
with the right-hand side of \eqref{6.4} as defined earlier. The desired 
generalization of Theorem~\ref{thm:6.1} is then as follows.

\begin{theorem}[\cite{DE11}, Theorem~21]\label{thm:6.2}
Let $b\geq 2$ and $\lambda\geq 1$ be integers, and $T$ and $Z$ as in 
Definition~\ref{def:2.3}. Then for $n\geq 0$ we have
\begin{equation}\label{6.5}
\omega_{b,T}^\lambda(n+1;Z) 
= \sum_{\substack{k_1+2k_2+\cdots+\lambda k_{\lambda}=n\\k_1,\ldots,k_\lambda\in M_b}}
\binom{k_1+\cdots+k_\lambda}{k_1,\ldots,k_\lambda}_b^*
z_1^{d_{t_1}^b(k_1)}\cdots z_\lambda^{d_{t_\lambda}^b(k_\lambda)}.
\end{equation}
\end{theorem}

We note that Theorem~\ref{thm:6.2} immediately implies 
Theorem~\ref{thm:6.1}. Indeed, when $b=2$, the multinomial coefficient defined
in \eqref{6.4} reduces to the one in \eqref{6.1} since $d^2_2(k_i) = k_i$. 
Then in the $b=2$ case we also have $d^b_{t_i}(k_i) = d_{t_i}(k_i)$, 
and the conditions concerning $M_b$ become irrelevant.

We are now ready to state and prove the main result of this section.

\begin{theorem}\label{thm:6.3}
Let $b\geq 2$ and $\rho\geq 1$ be integers, and let $\Lambda$, $T$, and $Z$ be
as in Definition~\ref{def:2.4}. Furthermore, let
\[
D_{\ell,i} := d^b_{t_{\ell,i}}(k_{\ell,i}),\qquad
1\leq i\leq\lambda_\ell \quad\hbox{and}\quad 1\leq\ell\leq\rho.
\]
Then for each integer $n\geq 0$ we have
\begin{equation}\label{6.6}
\Omega_{b,T}^\Lambda(n;Z) = \sum\prod_{\ell=1}^{\rho}
\binom{k_{\ell,1}+\cdots+k_{\ell,\lambda_\ell}}{k_{\ell,1},\ldots,k_{\ell,\lambda_\ell}}_b^*
z_{\ell,1}^{D_{\ell,1}}\cdots z_{\ell,\lambda_\ell}^{D_{\ell,\lambda_\ell}},
\end{equation}
where the sum is taken over all integers $k_{\ell,i}\in M_b$ with 
$1\leq i\leq\lambda_\ell$ and $1\leq\ell\leq\rho$, such that
\begin{equation}\label{6.7}
\sum_{\ell=1}^{\rho}\sum_{i=1}^{\lambda_\ell} i\cdot k_{\ell,i} = n.
\end{equation}
\end{theorem}

\begin{proof}
We combine Theorem~\ref{thm:2.6} with Theorem~\ref{thm:6.2}, where in 
\eqref{6.1} we replace $n$ by $n_\ell$, $\lambda$ by $\lambda_\ell$ and add
the additional subscript $\ell$ to $k_i$, $t_i$ and $z_i$, where 
$\ell=1,\ldots,\rho$. The right-hand side of \eqref{6.6} then follows 
immediately from the right-hand side of \eqref{2.8}, and the condition 
\eqref{6.7} follows from the summation conditions in \eqref{2.8} and in
\eqref{6.5}.
\end{proof}

We conclude this section with two examples. The first one is our ``running
example", continuing Examples~\ref{ex:2.5}, \ref{ex:3.1}, and \ref{ex:4.4}.

\begin{example}\label{ex:6.4}
{\rm Let $b=2$ and $\Lambda=(2,3)$, and to avoid double indices, we set again
$T=(s_1,s_2;t_1,t_2,t_3)$ and $Z=(y_1,y_2;z_1,z_2,z_3)$. We also set
$h_i:=k_{1,i}$, $i=1,2$, and $k_i:=k_{2,i}$, $i=1,2,3$. Then by 
Theorem~\ref{thm:6.3} we have
\begin{align}
\Omega_{2,T}^\Lambda(n;Z) 
=\sum_{\substack{h_1+2h_2+k_1\\+2k_2+3k_3=n}}
&\binom{h_1+h_2}{h_1,h_2}^*\binom{k_1+k_2+k_3}{k_1,k_2,k_3}^*\label{6.8} \\
&\times y_1^{d_{s_1}(h_1)}y_2^{d_{s_2}(h_2)}
z_1^{d_{t_1}(k_1)}z_2^{d_{t_2}(k_2)}z_3^{d_{t_3}(k_3)}.\nonumber
\end{align}
Since $b=2$, the condition in Theorem~\ref{thm:6.3} that concerns $M_b$ becomes
irrelevant; see the remark following Theorem~\ref{thm:6.2}.

As a specific example we consider $n=4$. By direct computation we find that the
solutions $(h_1,h_2;k_1,k_2,k_3)$ of $h_1+2h_2+k_1+2k_2+3k_3=4$ for which the
binomial/trinomial coefficients in \eqref{6.8} are odd (i.e., the ``starred"
versions are 1) are as follows:
\begin{align*}
&(0, 0; 4, 0, 0), (0, 0; 2, 1, 0), (0, 0; 0, 2, 0), (1, 0; 3, 0, 0), 
(1, 0; 0, 0, 1),\\
&(2, 0; 2, 0, 0), (2, 0; 0, 1, 0), (4, 0; 0, 0, 0), (3, 0; 1, 0, 0),
(0, 1; 2, 0, 0),\\
& (0, 1; 0, 1, 0), (2, 1; 0, 0, 0), (0, 2; 0, 0, 0),
\end{align*}
so that \eqref{6.8} becomes
\begin{align*}
\Omega_{2,T}^\Lambda(4;Z)
&= z_1^{d_{t_1}(4)}+z_1^{d_{t_1}(2)}z_2^{d_{t_2}(1)}+z_2^{d_{t_2}(2)}+
z_1^{d_{t_1}(3)}y_1^{d_{s_1}(1)}+z_3^{d_{t_3}(1)}y_1^{d_{s_1}(1)}\\
&\quad+z_1^{d_{t_1}(2)}y_1^{d_{s_1}(2)}+z_2^{d_{t_2}(1)}y_1^{d_{s_1}(2)}
+y_1^{d_{s_1}(4)}+z_1^{d_{t_1}(1)}y_1^{d_{s_1}(3)} \\
&\quad+z_1^{d_{t_1}(2)}y_2^{d_{s_2}(1)}+z_2^{d_{t_2}(1)} y_2^{d_{s_2}(1)}
+y_1^{d_{s_1}(2)}y_2^{d_{s_2}(1)} + y_2^{d_{s_2}(2)}\\
&=z_1^{t_1^2}+z_1^{t_1}z_2+z_2^{t_2}+z_1^{1+t_1}y_1
+z_3y_1+z_1^{t_1}y_1^{s_1}+z_2y_1^{s_1}+y_1^{s_1^2} \\
&\quad+z_1y_1^{1+s_1}+z_1^{t_1}y_2+z_2y_2+y_1^{s_1}y_2+y_2^{s_2}.
\end{align*}
This is consistent with the entry for $\Omega_{2,T}^{\Lambda}(4;Z)$ in Table~1.

It is interesting to note that the smallest $n$ for which all three entries
in the trinomial coefficient in \eqref{6.8} are nonzero is $n=11$, with 
$k_1=4$, $k_2=2$, and $k_3=1$, giving
\[
\binom{k_1+k_2+k_3}{k_1,k_2,k_3}=\binom{7}{4,2,1}=105\equiv 1\pmod{2}.
\]
This, in turn, leads to the following monomial and corresponding colored binary 
partition:
\[
z_{1}^{t_1^2} z_{2}^{t_2} z_3\;\longleftrightarrow \;
11 = 2_2^2 + 2_2 + 2_2 + 2_2^0 + 2_2^0 + 2_2^0.
\]
}
\end{example}

\begin{example}\label{ex:6.5}
{\rm We now choose $b=3$, but leave $\Lambda$, $T$, and $Z$ as in 
Example~\ref{ex:6.4}. Then in analogy to \eqref{6.8} we get
\begin{align}
\Omega_{3,T}^\Lambda(n;Z)
=\sum_{\substack{h_1+2h_2+k_1\\+2k_2+3k_3=n;\\h_1,\ldots,k_3\in M_3}}
&\binom{h_1+h_2}{h_1,h_2}_3^*\binom{k_1+k_2+k_3}{k_1,k_2,k_3}_3^*\label{6.9} \\
&\times y_1^{d_{s_1}^3(h_1)}y_2^{d_{s_2}^3(h_2)}
z_1^{d_{t_1}^3(k_1)}z_2^{d_{t_2}^3(k_2)}z_3^{d_{t_3}^3(k_3)}.\nonumber
\end{align}
This time we consider $n=6$ as a specific example. The conditions
$h_1,\ldots,k_3\in M_3$ mean, in particular, that 2 and 5 cannot occur among
the solutions of $h_1+2h_2+k_1+2k_2+3k_3=6$. By direct computation, using the
definition \eqref{6.4}, we find that both of the modified binomial/trinomial
coefficients in \eqref{6.9} are 1 when $(h_1,h_2;k_1,k_2,k_3)$ is of the form
\begin{align*}
& (0,3;0,0,0), (3,1;1,0,0), (3,0;3,0,0), (0,1;4,0,0),\\
& (4,0;0,1,0), (1,0;3,1,0), (0,0;0,3,0), (3,0;0,0,1), (0,0;3,0,1),
\end{align*}
so that \eqref{6.9} becomes
\begin{align*}
\Omega_{3,T}^\Lambda(6;Z) 
&= y_2^{d^3_{s_2}(3)}+y_1^{d^3_{s_1}(3)} y_2^{d^3_{s_2}(1)} z_1^{d^3_{t_1}(1)} 
+ y_1^{d^3_{s_1}(3)} z_1^{d^3_{t_1}(3)}+y_2^{d^3_{s_2}(1)}z_1^{d^3_{t_1}(4)}\\
&\quad +y_1^{d^3_{s_1}(4)} z_2^{d^3_{t_2}(1)} 
+ y_1^{d^3_{s_1}(1)}z_1^{d^3_{t_1}(3)}
z_2^{d^3_{t_2}(1)} + z_2^{d^3_{t_2}(3)}+ y_1^{d^3_{s_1}(3)}z_3^{d^3_{t_3}(1)}\\
&\quad+ z_1^{d^3_{t_1}(3)} z_3^{d^3_{t_3}(1)}  \\
& = y_2^{s_2}+y_1^{s_1} y_2 z_1+ y_1^{s_1} z_1^{t_1}+ y_2 z_1^{1+t_1}+
y_1^{1+s_1} z_2+ y_1 z_1^{t_1} z_2+ z_2^{t_2}\\
&\quad + y_1^{s_1}z_3+ z_1^{t_1} z_3.
\end{align*}
Finally we illustrate in Table~3 how the polynomial $\Omega_{3,T}^\Lambda(6;Z)$
translates to the set of nine $(2,3)$-restricted 2-colored ternary partitions
of $n=6$.
}
\end{example}

\begin{center}
\begin{table}[h]
{\renewcommand{\arraystretch}{1.3}
\begin{tabular}{|c|c||c|c|}
\hline
{\rm Monomial} & {\rm Partition} & {\rm Monomial} & {\rm Partition} \\
\hline
$y_{2}^{s_2}$ & $3_1+3_1$ & $y_{1} z_{1}^{t_{1}} z_{2}$ & $1_1+3_2+1_2+1_2$\\
$y_{1}^{s_1}y_{2} z_{1}$ & $3_1+1_1+1_1+1_2$ & $z_{2}^{t_2}$ & $3_2+3_2$ \\
$y_{1}^{s_1} z_{1}^{t_{1}}$ & $3_1+3_2$ & $y_{1}^{s_1}z_{3}$ & $3_1+1_2+1_2+1_2$   \\
$y_{2} z_{1}^{1+t_{1}}$ & $1_1+1_1+1_2+3_2$ & $z_{1}^{t_{1}}z_{3}$ & $3_2+1_2+1_2+1_2$   \\
$y_{1}^{1+s_1} z_{2}$  &  $1_1+3_1+1_2+1_2$ & & \\
\hline
\end{tabular}}
\medskip
\caption{The terms of $\Omega_{3,T}^\Lambda(6;Z)$ and corresponding partitions.}
\end{table}
\end{center}
\vspace{-9ex}

\section{Some product identities}

In dealing with polynomials in one or more variables with integer coefficients,
questions of divisibility and irreducibility become relevant. For instance, in
the case of multivariate $b$-ary Stern polynomials \cite{DE10}, we obtain an
analogue of Theorem~\ref{thm:1.7}. We present a special case of this result,
rewritten in the notation of the current paper. It corresponds to $\rho=1$ and
$\Lambda=(\lambda)=(b)$.

\begin{theorem}[\cite{DE10}, Corollary~4.3]\label{thm:7.1}
Let $b\geq 2$, $\ell\geq 1$, and $n\geq 1$ be integers. Then
\begin{equation}\label{7.1}
\Omega_{b,T}^{(b)}(n\cdot b^\ell+j;Z)
=\Omega_{b,T}^{(b)}(n;Z^{T^\ell})\cdot\Omega_{b,T}^{(b)}(j;Z),\qquad
\frac{b^\ell-1}{b-1}\leq j\leq b^\ell-1.
\end{equation}
\end{theorem}

In this section we show that Theorem~\ref{thm:7.1} can be extended to an 
arbitrary number of colors $\rho$ and an arbitrary 
$\Lambda=(\lambda_1,\ldots,\lambda_\rho)$. We use again the notations of 
Definition~\ref{def:2.4}.

\begin{theorem}\label{thm:7.2}
Let $b\geq 2$, $\ell\geq 1$, and $n\geq 1$ be integers. Then
\begin{equation}\label{7.2}
\Omega_{b,T}^\Lambda(n\cdot b^\ell+j;Z)
=\Omega_{b,T}^\Lambda(n;Z^{T^\ell})\cdot\Omega_{b,T}^\Lambda(j;Z),
\end{equation}
which in the case $b\leq\lambda=\lambda_1+\cdots+\lambda_\rho$ holds for all
integers $j$ with
\begin{equation}\label{7.3}
(\lambda-b+1)\frac{b^\ell-1}{b-1}\leq j\leq b^\ell-1.
\end{equation}
When $b>\lambda$, then \eqref{7.2} holds for all $j$ with 
$0\leq j\leq b^\ell-1$.
\end{theorem}

\begin{proof}
We proceed by induction on $\ell\geq 1$, and as usual we suppress the subscripts
and superscripts of $\Omega$. When $\ell=1$ and $\lambda-b+1\leq j\leq b-1$,
then the recurrence relation \eqref{4.4} gives
\begin{equation}\label{7.4}
\Omega(n\cdot b+j;Z) = Y_j\cdot\Omega(n;Z^T),
\end{equation}
since $b+j\geq\lambda+1$, so all other coefficients $Y_{bk+j}$ in \eqref{4.4}
vanish. For $n=0$, the identity \eqref{7.4} gives $\Omega(j;Z)=Y_j$, and we get
\begin{equation}\label{7.5}
\Omega(n\cdot b+j;Z) = \Omega(n;Z^T)\cdot \Omega(j;Z).
\end{equation}
In the case $\lambda<b$, \eqref{7.5} holds for all $0\leq j\leq b-1$. This 
establishes the induction beginning.

Now suppose that Theorem~\ref{thm:7.2} holds for some $\ell\geq 1$. We wish to
show that it also holds for $\ell+1$. To do so, we begin by setting
\begin{equation}\label{7.6}
j=r\cdot b+i,\qquad \lambda-b+1\leq i\leq b-1,\qquad
(\lambda-b+1)\frac{b^\ell-1}{b-1}\leq r\leq b^\ell-1.
\end{equation}
This means that $j$ satisfies the inequalities
\begin{align*}
j &\leq (b^\ell-1)b+(b-1) = b^{\ell+1}-1,\\
j &\geq (\lambda-b+1)\frac{b^\ell-1}{b-1}b+(\lambda-b+1)
= (\lambda-b+1)\frac{b^{\ell+1}-1}{b-1},
\end{align*}
as required. Now we have
\begin{align*}
\Omega(n\cdot b^{\ell+1}+j;Z) &= \Omega((n\cdot b^\ell+r)b+i;Z)\\
&= Y_i\cdot \Omega(n\cdot b^\ell+r;Z^T)\qquad\hbox{by\;}\eqref{7.4}\\
&= Y_i\cdot \Omega(n;Z^{T^{\ell+1}})\cdot\Omega(r;Z^T),
\end{align*}
where we have used the induction hypothesis \eqref{7.2}. Reordering this last
line, we have 
\[
\Omega(n\cdot b^{\ell+1}+j;Z) = \Omega(n;Z^{T^{\ell+1}})\cdot
Y_i\cdot\Omega(r;Z^T) 
= \Omega(n;Z^{T^{\ell+1}})\cdot\Omega(j;Z),
\]
where we have used \eqref{7.4} again. Note that in several steps the 
inequalities in \eqref{7.6} were essential. This completes the proof by 
induction.
\end{proof}

\begin{example}\label{ex:7.3}
{\rm We choose $b=4$, $\Lambda=(2,3)$, and leave $T$ and $Z$ as in 
Example~\ref{ex:6.4}. Using the recurrence relation in Theorem~\ref{thm:4.3}
we can compute $\Omega_{4,T}^{\Lambda}(n;Z)$ for $0\leq n\leq 11$, and factor
them when possible; see Table~4.

We note that the entries in Table~4 are consistent with Theorem~\ref{thm:7.2}.
The entries for $n=5, 8$, and 9 are irreducible and consist of 5, 7, and 8
monomials, respectively.}
\end{example}

\begin{center}
\begin{table}[h]
{\renewcommand{\arraystretch}{1.5}
\begin{tabular}{|l|l||l|l|}
\hline
$n$ & $\Omega_{4,T}^{\Lambda}(n;Z)$ & $n$ & $\Omega_{4,T}^{\Lambda}(n;Z)$ \\
\hline\hline
$0$ & $1$ & $6$ &  $(y_1^{s_1}+z_1^{t_1})(y_2+y_1 z_1+z_2)$  \\
\hline
$1$ & $y_1+z_1$ & $7$ & $(y_1^{s_1}+z_1^{t_1})(y_2z_1+y_1z_2+z_3)$ \\
\hline
$2$ & $y_2+y_1 z_1+z_2$ & $8$ & $y_2^{s_2}+\cdots+y_1z_1^{t_1}z_3$ \\
\hline
$3$ & $y_2 z_1+y_1 z_2+z_3$ & $9$ & $y_1y_2^{s_2}+\cdots+y_2z_1^{t_1}z_3$  \\
\hline
$4$ & $y_1^{s_1}+z_1^{t_1}+y_2 z_2+y_1 z_3$ 
& $10$ & $(y_2^{s_2}+y_1^{s_1}z_1^{t_1}+z_2^{t_2})(y_2+y_1 z_1+z_2)$    \\
\hline
$5$ & $y_1^{1+s_1}+\cdots +y_2 z_3$ 
& $11$ & $(y_2^{s_2}+y_1^{s_1}z_1^{t_1}+z_2^{t_2})(y_2 z_1+y_1 z_2+z_3)$ \\
\hline
\end{tabular}}
\medskip
\caption{$\Omega_{4,T}^{\Lambda}(n;Z)$ with $\Lambda=(2,3)$, for
$0\le n \le 11$.}
\label{tab:4}
\end{table}
\end{center}
\vspace{-4ex}

%

As far as Theorem~\ref{thm:1.7} is concerned, we believe that it extends in 
this form to general $\rho\geq 1$ and $\Lambda=(\lambda_1,\ldots,\lambda_\rho)$
only when $1\leq(\lambda+1)/b<2$, that is, when 
$b-1\leq\lambda_1+\cdots+\lambda_\rho\leq 2b-2$. This would be the case, for 
instance, in Example~\ref{ex:7.3}. Pursuing this question further would be 
beyond the scope of this paper.

\section{Further Remarks}

We recall that by setting $Z=(1,\ldots,1)$, we get the
numerical sequences $C_b^\Lambda(n)=\Omega_{b,T}^\Lambda(n;(1,\ldots,1))$; see
\eqref{2.5}. While it is not the purpose of this paper to study these sequences,
we present one easy but rather surprising result.

\begin{corollary}\label{cor:7.4}
Let $b\geq 2$ be an integer, and let $\Lambda=(b-1,\ldots,b-1)$ with 
$\rho\geq 1$ elements. Then the number of $\Lambda$-restricted $\rho$-colored
$b$-ary partitions is independent of $b$, namely
\begin{equation}\label{7.7}
C_b^\Lambda(n) = \binom{n+\rho-1}{\rho-1},\qquad n=0,1,\ldots
\end{equation}
\end{corollary}

\begin{proof}
We use the identity
\begin{equation}\label{7.8}
\prod_{j=0}^\infty\bigg(1+q^{b^j}+q^{2\cdot b^j}\cdots+q^{(b-1)\cdot b^j}\bigg)
= \sum_{n=0}^\infty q^n = \frac{1}{1-q},
\end{equation}
which reflects the fact that each integer $n\geq 0$ has a unique $b$-ary
expansion, or, in other words, $S_b^{b-1}(n)=1$. Now we raise
the left- and the right-most terms in \eqref{7.8} to the power $\rho$ and use
the well-known binomial expansion 
\[
(1-q)^{-\rho} = \sum_{n=0}^\infty\binom{n+\rho-1}{\rho-1}q^n
\]
(see, e.g., \cite[Eq.~(1.3)]{Go}). Then the generating function \eqref{2.1}
immediately gives the desired identity \eqref{7.7}.
\end{proof}

\begin{example}\label{ex:7.5}
{\rm Let $\rho=3$ and $n=3$. When $b=2$, we have $\Lambda=(1,1,1)$, and the 
$\binom{5}{2}=10$ allowable partitions are}

$2_3+1_3,\;2_3+1_2,\;2_3+1_1,\;2_2+1_3,\;2_2+1_2,\;2_2+1_1,\;2_1+1_3,\;
2_1+1_2,\;2_1+1_1,$

\quad $1_3+1_2+1_1.$

\noindent
{\rm For $b=3$, and thus $\Lambda=(2,2,2)$, we have}

$3_3,\;3_2,\;3_1,\;1_3+1_3+1_2,\;1_3+1_3+1_1,\;1_3+1_2+1_2,\;1_3+1_2+1_1,$

\quad $1_3+1_1+1_1,\;1_2+1_2+1_1,\;1_2+1_1+1_1,$

\noindent
{\rm and for any $b\geq 4$ we have the partitions}

$1_3+1_3+1_3,\;1_3+1_3+1_2,\;\ldots,\;1_2+1_1+1_1,\;1_1+1_1+1_1.$

\noindent
{\rm This last example is an instance of the more general situation where 
$n<b$, in which case all partitions have parts 1 with $\rho$ different colors.
These colors may be assigned $b-1\geq n$ times, which means that there
are no restrictions. Therefore we have a one-to-one correspondence between 
such partitions and all non-increasing sequences of length $n$ consisting of
integers $1, 2,\ldots, \rho$. The number of such sequences is a known
combinatorial quantity, namely $\binom{n+\rho-1}{\rho-1}$, consistent with
Corollary~\ref{cor:7.4}.}
\end{example}

Finally we return to the main topic of this paper, namely the multivariate 
polynomials that characterize all $\Lambda$-restricted $\rho$-colored $b$-ary
partitions of an integer $n\geq 1$.

\begin{example}\label{ex:7.6}
{\rm As in Example~\ref{ex:7.5} we let $\rho=3$ and $n=3$. Using any of the 
various methods described in this paper for computing the relevant polynomials,
we obtain the following expressions.

\smallskip
\noindent
(a) Let $b=2$ and $\Lambda=(1,1,1)$. Then with $Z=(x_1; y_1;z_1)$  and 
$T=(r_1;s_1;t_1)$ we get
\begin{align}
\Omega_{2,T}^{(1,1,1)}(3;Z) &= z_1^{1+t_1}+y_1z_1^{t_1}+x_1z_1^{t_1}+
y_1^{s_1}z_1 + y_1^{1+s_1}\label{7.9} \\
&\quad + x_{1} y_{1}^{s_{1}} + x_{1}^{r_{1}} z_{1}+ x_{1}^{r_{1}} y_{1} + 
x_{1}^{1 + r_{1}} + x_{1} y_{1} z_{1}.\nonumber
\end{align}

\smallskip
\noindent
(b) Let $b=3$ and $\Lambda=(2,2,2)$. Then with 
$Z=(x_1,x_2; y_1,y_2;z_1,z_2)$ and $T=(r_1,r_2; s_1,s_2;t_1,t_2)$ we get
\begin{align}
\Omega_{3,T}^{(2,2,2)}(3;Z) &= z_1^{t_1}+y_1^{s_1}+x_1^{r_1}+y_1 z_2+x_1 z_2\label{7.10}\\
&\quad +  y_2 z_1+x_1 y_1 z_1+x_2 z_1+x_1 y_2+x_2 y_1. \nonumber
\end{align}

\smallskip
\noindent
(c)  Let $b\geq 4$ and $\Lambda=(b-1,b-1,b-1)$. Then with     
$Z=(x_1,\ldots,x_{b-1}; y_1,\ldots,y_{b-1};$ $z_1,\ldots,z_{b-1})$ we get
\begin{align}
\Omega_{b,T}^{\Lambda}(3;Z) &= z_{3}+y_{1}z_{2}+x_{1}z_{2}+y_{2}z_{1} + 
x_{1}y_{1}z_{1}\label{7.11} \\
&\quad + x_{2}z_{1} +  y_{3} + x_{1} y_{2} + x_{2} y_{1} + x_{3}.\nonumber
\end{align}
Here in \eqref{7.11} we did not specify $T$ because no exponents other than 1 
occur.

All these polynomials \eqref{7.9}--\eqref{7.11} consist of ten monomials, which
is consistent with Corollary~\ref{cor:7.4} and Example~\ref{ex:7.5}. 
Each of these polynomials characterizes the three sets of partitions listed 
in Example~\ref{ex:7.5}, in the order in which the respective monomials are
given.}
\end{example}


\begin{thebibliography}{25}

\bibitem{An} G.~E.~Andrews, Congruence properties of the $m$-ary partition 
function, {\it J. Number Theory} {\bf 3} (1971), 104--110. 

\bibitem{BM} B.~Bates and T.~Mansour, The $q$-Calkin-Wilf tree,
{\it J. Combin. Theory Ser. A} {\bf 118} (2011), 1143--1151.

\bibitem{Ch} R.~F.~Churchhouse, Congruence properties of the binary partition 
function, {\it Proc. Cambridge Philos. Soc.} {\bf 66} (1969), 371--376.

\bibitem{CL} S.~Corteel and J.~Lovejoy, Overpartitions,
{\it Trans. Amer. Math. Soc.} {\bf 356} (2004), no.~4, 1623--1635.

\bibitem{DE7} K.~Dilcher and L.~Ericksen, Generalized Stern polynomials and 
hyperbinary representations, {\it Bull. Pol. Acad. Sci. Math.} {\bf 65} (2017),
11--28.

\bibitem{DE9} K.~Dilcher and L.~Ericksen, Polynomials characterizing hyper 
$b$-ary representations, {\it J. Integer Seq.} {\bf 21} (2018), Article 18.4.3,
11 pp.

\bibitem{DE11} K.~Dilcher and L.~Ericksen, Polynomial analogues of restricted
$b$-ary partition functions, {\it J. Integer Seq.} {\bf 22} (2019), Article
19.3.2.

\bibitem{DE10} K.~Dilcher and L.~Ericksen, Properties of multivariate b-ary
Stern polynomials, {\it Ann. Comb.}, to appear.

\bibitem{DST} J.~M.~Dumont, N.~Sidorov, and A.~Thomas, Number of representations
related to a linear recurrent basis, {\it Acta Arith.} {\bf 88} (1999), no.~1,
371--396.

\bibitem{Eu} L.~Euler, De partitione numerorum, {\it Novi Commentarii Academiae
Scientiarum Petropolitanae} {\bf 3} (1753), 125--169. In {\it Opera Omnia},
Series 1, Volume 2, pp.~254--294. Also available from {\it The Euler Archive},
{\tt http://eulerarchive.maa.org}, Paper E191.

\bibitem{Go} H.~W.~Gould, {\it Combinatorial Identities\/},
revised edition, Gould Publications, Morgantown, W.Va., 1972.

\bibitem{Ke} W.~J.~Keith, Restricted $k$-color partitions, {\it Ramanujan J.}
{\bf 40} (2016), 71--92.

\bibitem{Ma} K.~Mahler, On a special functional equation,
{\it J. London Math. Soc.} {\bf 15} (1940), 115--123.

\bibitem{Re} B.~Reznick, Some binary partition functions, in {\it Analytic
Number Theory: Proceedings of a Conference in Honor of Paul T.~Bateman},
(B.~C.~Berndt et al., Eds.), Birkh\"auser, Boston, 1990, 451--477.

\bibitem{Ro} \"O.~R\"odseth, Some arithmetical properties of $m$-ary 
partitions, {\it Proc. Cambridge Philos. Soc.} {\bf 68} (1970), 447--453.

\bibitem{RS} {\O}.~R{\o}dseth and J.~A.~Sellers, On m-ary overpartitions,
{\it Ann. Comb.} {\bf 9} (2005), no.~3, 345--353. 

\bibitem{SW} R.~P.~Stanley and H.~S.~Wilf, Refining the Stern diatomic sequence,
{\it Preprint}, 2010, \\
\tt{http://www-math.mit.edu/\~{}rstan/papers/stern.pdf}.\rm

\bibitem{UZ} M.~Ulas and B.~{\.Z}mija, On $p$-adic valuations of colored 
$p$-ary partitions, {\it Monatsh. Math.} {\bf 188} (2019), no.~2, 351--368. 

\end{thebibliography}
\end{document}